\documentclass[11pt, a4paper, UKenglish]{article}

\usepackage[UKenglish]{babel}
\usepackage[T1]{fontenc}
\usepackage[utf8]{inputenc}
\usepackage{amsmath,amssymb,amsfonts,amsthm, mathtools, mathrsfs}
\usepackage{parskip, color}
\usepackage{fancyhdr}
\usepackage{dsfont}
\usepackage{graphicx, pgf}
\usepackage{caption}
\usepackage{float}
\usepackage{subcaption}
\usepackage{enumerate}
\usepackage{hyperref}
\usepackage{cleveref}
\usepackage[left=2.5cm, right=2.5cm]{geometry}

\newtheorem{theorem}{Theorem}
\newtheorem{lemma}[theorem]{Lemma}
\newtheorem{corollary}[theorem]{Corollary}
\newtheorem{proposition}[theorem]{Proposition}
\theoremstyle{definition}
\newtheorem{definition}{Definition}
\newtheorem{remark}[theorem]{Remark}

\def\Kon{\ensuremath{K_{on}}}
\def\Koff{\ensuremath{K_{off}}}
\def\Kfb{\ensuremath{K_{fb}}}
\def\floorxt{\ensuremath{\lfloor tx \rfloor + 1}}

\def\to{\rightarrow}
\def\To{\longrightarrow}

\newcommand{\BV}{\operatorname{BV}}
\newcommand{\weakstar}{\ensuremath{\xrightharpoonup{*}}}
\newcommand{\N}{\ensuremath{\mathbb{N}}}   % natuerliche Zahlen
\newcommand{\Ctj}{\ensuremath{C^{t_{j}}}}
\newcommand{\piphi}{\ensuremath{\int_{(k-1)t_{j}^{-1}}^{kt_j^{-1}} \varphi(s, x) \text{ d}x }}
\newcommand{\Sobolev}{\ensuremath{W^{2,2}((1 + x)^2 \text{ d}x, (0, \infty))}}

\author{Lorena Pohl\thanks{Universität Bonn, Germany, \texttt{pohl@iam.uni-bonn.de}} \hspace*{3 cm} Barbara Niethammer\thanks{Universität Bonn, Germany, \texttt{niethammer@iam.uni-bonn.de}}}
\title{A Becker-D\"{o}ring type model for cell polarization}
\date{August 1st, 2023}

\begin{document}
	\maketitle

{\small \textbf{Abstract}
We propose a model for cell polarization based on the Becker-D\"{o}ring equations with the first coagulation coefficient equal to zero. We show convergence to equilibrium for power-law coagulation and fragmentation rates and obtain a loss of mass in the limit $t \to \infty$ depending on the initial mass and the relative strengths of the coagulation and fragmentation processes. In the case of linear rates, we further show that large clusters evolve in a self-similar manner at large times by comparing limits of appropriately rescaled solutions in different spaces. 

\textbf{Acknowledgments}
The authors gratefully acknowledge the financial support of the Deutsche Forschungsgemeinschaft (DFG, German Research Foundation) through the collaborative research
centre The mathematics of emerging effects (CRC 1060, Project-ID 211504053) and the Bonn International Graduate School of Mathematics at the Hausdorff Center for Mathematics (EXC 2047/1, Project-ID 390685813).

This version of the article has been accepted for publication after peer review but is not the Version of Record and does not reflect post-acceptance improvements or any corrections. The Version of Record is available online at \url{https://doi.org/10.1007/s10955-023-03144-0}

\textbf{Keywords} \textit{Becker-D\"{o}ring equations}, \textit{self-similarity}, \textit{long-time behaviour}}

\section{Introduction}\label{sec1}
In this paper we analyse the long-time behaviour of a variant of the classical Becker-D\"{o}ring model which is motivated by models for cell polarization. 
Cell polarization is a biological process in which a signalling molecule accumulates on a discrete part of a cell's membrane. While the area of localization is often predetermined by pre-existing spatial cues, many cell types can spontaneously establish cell polarity in random orientations \cite{wedlich-soldner-2004}.

For example, in the budding yeast \textit{Saccharomyces cerevisiae}, proteins governed by the protein Cdc42 form a cap on the membrane, thereby determining the site of a new bud. While the mechanisms driving this process are complex and not completely understood, it has been shown that Cdc42-GTP recruits its activator to the membrane, causing a positive feedback loop  \cite{kluender2013,martin2014}.

In  \cite{altschuler2008}, the authors propose two models for cell polarization relying on such a feedback loop, one deterministic and one stochastic. The deterministic model describes a finite number of particles distributed in the cytosol and on the membrane of a cell and moving between the two according to three processes: A molecule in the cytosol can spontaneously move to a random position on the membrane with the rate $\Kon$, a molecule on the membrane can switch back to the cytosol with the rate $\Koff$, and a molecule in the cytosol can be recruited by a feedback loop to a region of the membrane where the concentration of particles is already high with a rate proportional to $\Kfb$ and the fraction of total molecules in the cytosol. Additionally, particles on the membrane are subject to diffusion. For experimental estimation of the rates of these processes, see for instance \cite{wedlich-soldner2007}.

The authors of \cite{altschuler2008} find that solutions to the deterministic model cannot maintain a polarized state, and if the initial state of the system is nearly spatially homogeneous, a polarized state will not appear. This leads the authors to consider a stochastic model instead, for which they find different behaviour depending on the relative strengths of the feedback mechanism and the spontaneous transition of particles from the cytosol to the membrane. They find that when $\frac{\Kon}{\Kfb}$ is very small compared to the number of signalling molecules, a single distinct polarized region appears, whereas no localization takes place when $\frac{\Kon}{\Kfb}$ is large.

Here we propose a different deterministic model for cell polarization which allows for a polarized state to appear independently of the spatial homogeneity of the initial state. We are inspired by existing models for coagulation-fragmentation processes to describe the accumulation of particles at a certain site by the formation of a cluster. We denote the concentration of clusters of size $n \in \N$ as $c_n(t)$ and the concentration of single molecules or monomers in the cytosol as $f(t)$. Cluster formation and fragmentation are driven by the same processes as in \cite{altschuler2008}: a feedback loop by which clusters recruit molecules from the cytosol and spontaneous movement of monomers between the cytosol and the membrane. 
This leads to considering the system
\begin{align}\label{main-equations}
\begin{split}
\partial_t c_n &= J_{n-1} - J_n, \quad n \geq 1,\\
\partial_t f &= - \sum_{n = 0}^{\infty} J_n,
\end{split}
\end{align}
where the fluxes $J_n$ are defined as
\begin{align}\label{def_fluxes_j_n-a}
\begin{split}
J_n &= K_{fb} a_n f c_n - K_{off} b_{n+1} c_{n + 1}, \quad n \geq 1 \\
J_0 &= K_{on} f  - K_{off} c_1,
\end{split}
\end{align}
We fix the initial mass of the system to be
\begin{equation}\label{rho-is-initial-mass}
\rho \coloneqq f(0) + \sum_{n = 1}^\infty nc_n(0) < \infty.
\end{equation} 

In this model, the emergence of cell polarity can be described by the formation of  super-clusters of extremely large size as $t \to \infty$. This occurs when some or all of the mass is transported into clusters of larger and larger size as time goes on, so that the total system loses mass in the limit $t \to \infty$.

Since the authors of \cite{altschuler2008} find that polarization occurs in their model only when $\frac{\Kon}{\Kfb}$ is very small we suppose in the following that $K_{on}=0$ and combine the remaining two coefficients in a new constant $\kappa = \frac{\Kfb}{\Koff}$. 
This leads to 
\begin{align}\label{def_fluxes_j_n}
\begin{split}
J_n &=  \kappa a_n f c_n -  b_{n+1} c_{n + 1}\,, \quad n \geq 1 \\
J_0 &=  - c_1\,.
\end{split}
\end{align}

We remark that while we consider this model in the context of cell polarization, further investigation regarding connections to problems in other areas such as protein aggregation or vesicular transport in neurons would be interesting, as our model has similarities to those considered e.g. in \cite{prigent_2012, xue_2008, Bressloff_2016}.
Similarly, it would be interesting to consider a variation of our model including spatial diffusion of clusters on the membrane.

System \eqref{main-equations}+\eqref{def_fluxes_j_n} is a variation of the Becker-D\"{o}ring equations, which were first treated in depth in \cite{ball-carr-penrose}.  The Becker-D\"{o}ring equations describe the dynamics of clusters of particles where clusters can only gain or lose a single particle at a time. The equations are given by
\begin{align}
\partial_t c_n &= J_{n-1} - J_n, \quad n \geq 2 \label{becker-doring-eqns1} \\
\partial_t c_1 &= -J_1 -  \sum_{n = 1}^{\infty} J_n, \label{becker-doring-eqns2}
\end{align}
with fluxes
$J_n = a_n c_1 c_n - b_{n+1}c_{n+1},$
where $a_n > 0$ are the coagulation and $b_n > 0$ the fragmentation coefficients. 

Well-posedness of \eqref{main-equations}+\eqref{def_fluxes_j_n} and the fact that solutions are strictly positive for all times assuming $c_{n}(0) > 0$ for at least one $n \in \N$ can be shown analogously to the standard Becker-D\"{o}ring equations, see \cite{ball-carr-penrose, laurencot-mischler}. We state the result here for completeness.

\begin{proposition}\label{P.existence}
	Suppose that $0<a_n,b_n=O(n)$ as $n\to \infty$ and that we have nonnegative initial data $f(0)$ and $(c_n(0))_{n\in\N}$ such that
	$\rho:=f(0)+\sum_{n=1}^{\infty} n c_n(0)< \infty$. Then there exists a global nonnegative solution to \eqref{main-equations}, \eqref{def_fluxes_j_n}  that conserves the mass, i.e. it holds
	$\rho=f(t)+ \sum_{n=1}^{\infty} nc_n(t)$ for all $t \in [0,\infty)$.
\end{proposition}

We remark that the assumption $a_n,b_n=O(n)$ is not too restrictive for our purposes, since we do not expect superlinear growth and fragmentation rates in the biological application. An interesting feature of the Becker-D\"{o}ring equations lies in their long-time behaviour, which depends on whether or not the initial mass of the system exceeds a critical density that depends on the coefficients $a_n$ and $b_n$.  Setting $\partial_t c_n = 0$ in \eqref{becker-doring-eqns1} and \eqref{becker-doring-eqns2} implies that $J_n = 0$ for all $n \geq 1$. Therefore, the equilibrium solutions have the form $c_n = Q_nc_1^n$, where $Q_1 = 1$ and
\begin{equation}\label{bd-equilibria-shape}
Q_n = \frac{\prod_{i = 1}^{n-1} a_i}{\prod_{i = 2}^{n} b_i}
\end{equation}
for $n \geq 2$. Hence, the equilibrium solutions are fully characterized by their value for $c_1$, and their mass is given by
$\sum_{n = 1}^{\infty} nQ_nc_1^n.$
The radius of convergence of this sum is given by \linebreak
$z \coloneqq \big(\limsup\limits_{n \to \infty} Q_n^{\frac{1}{n}} \big)^{-1} \in [0, \infty) \cup \{\infty \}$
and the corresponding critical mass by $\rho_z = \sum_{n = 1}^{\infty} nQ_nz^n \in [0, \infty) \cup \{\infty \}.$
It can be shown that for initial masses $\rho \leq \rho_z < \infty$, respectively $\rho < \infty$ when $\rho_z = \infty$, solutions converge to the unique equilibrium with mass $\rho$ as $t \to \infty$ . However, when the initial mass of the system exceeds the critical mass $\rho_z$, there is no equilibrium state with mass $\rho$. In this case, it has been shown that $c_n(t) \to Q_n z^n$ as $t \to \infty$ for all $n \geq 1$, meaning that the solution converges weakly to the equilibrium with the critical mass $\rho_z$, and the excess mass $\rho - \rho_z$ is transported to larger and larger clusters as $t \to \infty$, disappearing in the limit \cite{ball-carr-penrose, Slemrod_1989, carr-dunwell}.

In contrast to the standard Becker-D\"{o}ring model, our model includes two types of monomers, those on the membrane and those in the cytosol. Furthermore, we do not allow monomers in the cytosol to spontaneously pass onto the membrane. This can be compared to setting the coagulation rate $a_1 = 0$ in the regular Becker-D\"{o}ring equations, effectively breaking the chain of coagulation. In the existing literature on the Becker-D\"{o}ring equations, all coagulation and fragmentation rates are generally assumed to be strictly positive, and the existing methods used to establish the long-time behaviour of the system are no longer applicable to our model since they rely on the shape of equilibria given by \eqref{bd-equilibria-shape} and the existence of a Lyapunov-functional.
However, equilibria of the system \eqref{main-equations}+\eqref{def_fluxes_j_n} have the form $c_n = 0$ for all $n \geq 1$. The only mass is in $f$, which means that equilibria of any finite mass exist. This begs the question of whether or not mass ever vanishes in the limit $t \to \infty$ like in the regular Becker-D\"{o}ring equations. We answer this question for power-law coefficients $a_n = b_n = n^{\lambda}$ for $\lambda \in [0,1]$.
\begin{proposition}\label{P.longtime}
	Let $a_n=b_n=n^{\lambda}$ for some $\lambda \in [0,1]$ and assume $f(0) < \rho$. 
	\begin{enumerate}
		\item \label{P.longtimea}
		Then the solution of \eqref{main-equations}+\eqref{def_fluxes_j_n}satisfies $c_n(t) \to 0$ as  $t \to \infty$ for all $n \geq 1$. 
		\item \label{P.longtimeb}
		Suppose  that $\sum_{n=1}^{\infty} n^{2-\lambda } c_n(0)<\infty$. Then 
		$f(t)\to \min (\rho,\frac{1}{\kappa})$ as $ t \to \infty$.
	\end{enumerate}
\end{proposition}

In particular, this result implies that there is indeed loss of mass in the limit $t\to \infty$ if $\rho$ is larger than a critical mass, which in our model is $\rho_{crit}=\frac{1}{\kappa}$. The proof of Proposition \ref{P.longtime} is the content of Section \ref{sec-longtime-an=bn=nlambda}.
The assumption $f(0) < \rho$, which in our application implies that some mass is initially present on the membrane, is made to exclude the trivial case $f(0) = \rho = f(t)$ for all $t \geq 0$.

\begin{remark}
	The behaviour in Prop. \ref{P.longtime} differs from the case where $\Kon > 0$, which is much more closely related to the classical Becker-D\"{o}ring equations and can be treated by the same methods. Indeed if we had $J_0 = \Kon f - c_1$ in \eqref{def_fluxes_j_n}, equilibria of the system would have the shape $c_n = Q_n \Kon \kappa^{n-1}f^n$ and mass
	$\rho = f + \Kon \sum_{n = 1}^\infty nQ_n \kappa^{n-1}f^n$.
	Long-time behaviour of the system would be as described above for the regular Becker-D\"{o}ring system. In particular, for $a_n = b_n = n^{\lambda}, \lambda \in [0,1]$, the critical monomer density $f$ is still $\frac{1}{\kappa}$, but the corresponding critical mass is infinity, meaning that equilibria of any finite mass exist and the system converges as $t \to \infty$ to the unique equilibrium with the initial mass. 
	
	We see that for given mass $\rho$ and very small $\Kon$ the equilibrium value of $f$ must be close to $\frac{1}{\kappa}$ and most of the mass sits in the larger clusters. Hence, one might expect that for a finite time horizon the system for $\Kon$ positive but small evolves similarly to the one for $\Kon=0$ even though eventually in the long-time limit the mass escapes to infinity if $\Kon=0$ and hence the solutions are not close anymore.
	
	In the case of $\lambda = 1$, it is possible to explicitly compute the value of $f$ in the limit since the equation for $\partial_t f$ is an ODE, which is not the case in the typical Becker-D\"{o}ring equations, and one can check that this limit value is indeed the one for which the equilibrium solution has the correct mass.
	
	Finally, let us remark that is possible to construct physically reasonable coagulation and fragmentation coefficients such that a loss of mass takes place also for the $\Kon > 0$ system analogously to the classical Becker-D\"{o}ring equations, however, these are generally more complex to treat for $\Kon = 0$.
\end{remark}

Another important question for the classical Becker-D\"{o}ring equations
is how the excess mass contained in the large clusters evolves for large times. It has been shown in \cite{Penrose97,Vel98,laurencot-mischler,niethammer2003} that the evolution of these large clusters can be described, for typical coefficients in certain phase transition processes,  by the classical LSW mean-field model for coarsening. For the latter solutions typically evolve in a self-similar fashion, but it has so far been out of reach to prove  full self-similarity of the evolution of the large clusters in the Becker-D\"{o}ring equations.

It is the main goal of this paper to establish such a self-similarity result for system \eqref{main-equations}+ \eqref{def_fluxes_j_n} with coefficients $a_n=b_n=n$. More precisely, we show  in Section \ref{sec-selfsim} that the appropriately rescaled tails $C_k(t):=\sum_{n=k}^{\infty}c_n(t)$ converge to a self-similar profile as $t \to \infty$. 

\begin{theorem}\label{T.main}
	Suppose that $a_n=b_n=n$. 
	Then it holds that 
	\begin{equation}\label{main1}
	tC_{\lfloor t x \rfloor + 1}(t) \weakstar  \Big(\rho - \frac{1}{\kappa}\Big)e^{-x} \qquad \mbox{ in } L^\infty([0, \infty)) \mbox{ as } t \to \infty\,.
	\end{equation}
	Furthermore
	let $0 < a < T < \infty$. For all $1 \leq p < \infty$ and $M < \infty$ it holds that
	\begin{equation}\label{main2}  tC_
	{\lfloor t x \rfloor + 1}(ts) \to \Big( \rho - \frac{1}{\kappa}\Big) \frac{1}{s}e^{-\frac{x}{s}}
	\end{equation}
	in $C([a,T], L^p((0,M)))$ as $t \to \infty$.
\end{theorem}

In particular, this suggests that the typical size of a large cluster on the membrane grows approximately at rate $t$, a fact that might also be of interest for biological applications. It would be interesting to also examine the speed of convergence towards the self-similar profile, however, this is outside the scope of this paper. 

We expect that a comparable self-similarity result should also hold for the power-law coefficients with $\lambda \in [0,1)$. However, the proof of 
Theorem \ref{T.main} relies on the fact that we have in this case an explicit ODE for $f$ which implies exponentially fast convergence to the equilibrium value $\frac{1}{\kappa}$. This allows us to adapt a strategy of 
\cite{eichenberg-schlichting} where self-similar long-time behaviour is proved for a model of exchange-driven growth. 

Finally, let us remark that for the power-law coefficients considered here we do not expect to have metastable long-time behaviour if $\rho-\frac{1}{\kappa}$ is small. This is clear for $\lambda =1$ since then we obtain exponentially fast decay of $f$ towards its equilibrium value (even though the rate becomes small if $\rho-\frac{1}{\kappa}$ is small), but also for $\lambda<1$ one would not expect such a phenomenon. For non-power law coefficients one might obtain metastability, however, just as in the case of the classical BD equations.

Throughout the paper, we use $A \lesssim B$ to denote that $A \leq DB$ for a constant $D$ that is independent of time.

\section{Proof of Proposition \ref{P.longtime}}\label{sec-longtime-an=bn=nlambda}

We assume that $f(0)<\rho$, since otherwise no dynamics take place. 
Our strategy for showing the convergence of $f(t)$ and $c_n(t)$ will depend on the value of $f(0)$. We find that 
\begin{align*}
\partial_t f &= - \sum_{n = 0}^{\infty} J_n = c_1 - \sum_{n = 1}^{\infty} \kappa n^\lambda f c_n + \sum_{n = 2}^{\infty} n^\lambda c_n 
%= -\kappa f \sum_{n = 1}^\infty n^\lambda c_n + \sum_{n = 1}^\infty n^\lambda c_n 
\\&= (1 - \kappa f) \sum_{n = 1}^{\infty} n^\lambda c_n 
\begin{cases}
>0, & \text{if } f \in [0, \frac{1}{\kappa }), \\
=0, & \text{if } f = \frac{1}{\kappa }, \\
<0, & \text{if } f \in (\frac{1}{\kappa}, \rho),
\end{cases} 
\end{align*}
hence $f$ is monotonously increasing if $f(0) < \frac{1}{\kappa}$ and decreasing if $f(0) > \frac{1}{\kappa}$.
We begin by examining the long-time behaviour of clusters on the membrane. 

\subsection{Monomers in the cytosol: proof of Prop. \ref{P.longtime}, part \ref{P.longtimea}}

{\bf Case I:   $f(0) \leq \frac{1}{\kappa}.$} We have seen that then $f(t) \leq \frac{1}{\kappa}$ for all $t \geq 0$, so transport to larger clusters should be slower than in the case of pure diffusion. Thus we expect to find a supersolution by considering
\begin{align}\label{pure-diffusion-system}
\partial_t \tilde{c}_n &= \tilde{J}_{n-1} - \tilde{J}_n\,, \qquad \mbox{ where } \quad \tilde{J}_n= n^\lambda \tilde{c}_n - (n+1)^\lambda \tilde{c}_{n + 1}\,, \quad n \geq 1\,,
\end{align}
with initial data $\tilde{c}_n(0) = c_n(0)$ and fixed $\tilde{c}_0 \equiv 0$. Note that we use the convention $0^0 = 1$. 

We will show that $\sum_{n = k}^{\infty} c_n(t) \leq \sum_{n = k}^{\infty} \tilde{c}_n(t)$ for all $k$ and that $\sum_{n=1}^{\infty} \tilde c_n(t) \to 0$ as $t \to \infty$, from which the desired result follows.
Indeed, if we define 
$C_k = \sum_{n = k}^{\infty} c_n,$
$\tilde{C}_k = \sum_{n = k}^{\infty} \tilde{c}_n$
and $E_k = C_k - \tilde{C}_k$
then 	\begin{align*}
\partial_t C_k &= \sum_{n = k}^{\infty} \partial_t c_n 
= \sum_{n = k}^{\infty} \left(J_{n-1} - J_n\right) = J_{k-1},
\end{align*}
and similarly, 
$\partial_t \tilde{C}_k = \tilde{J}_{k-1}.$
Since $f(0) \leq \frac{1}{\kappa}$ implies that $\kappa f(t) - 1 \leq 0$ for all $t \geq 0$, we have for $k \geq 2$ that
\begin{align*}
\partial_t E_k &= J_{k-1} - \tilde{J}_{k-1} \\
&= (\kappa f - 1)(k-1)^\lambda c_{k-1} +  (k-1)^\lambda (c_{k-1} - \tilde{c}_{k-1}) + k^\lambda(\tilde{c}_{k} - c_{k}) \\
&\leq (k-1)^\lambda (c_{k-1} - \tilde{c}_{k-1}) + k^\lambda(\tilde{c}_{k} - c_{k}) \\
&=(k-1)^\lambda(E_{k-1} - E_{k}) + k^{\lambda} (E_{k+1} - E_{k}).
\end{align*}
For $k = 1$ we have
$$\partial_t E_1 = -c_1 + \tilde{c}_1 = -E_1 + E_2.$$ 
Since $C_k \to 0$ and $\tilde{C}_k \to 0$ uniformly on $[0, T]$ as $k \to \infty$, we can conclude by standard comparison arguments that $E_k \leq 0$ for all $k$.  (See also for example \cite[Prop. 2.8]{canizo-uniform-moments}.)

We remark that conservation of mass also holds for the system (\ref{pure-diffusion-system}).
Therefore, from \linebreak $\tilde{c}_n (0) = c_n(0)$ we obtain that $\tilde{c}_n (t) \leq \frac{1}{n}$. Furthermore, as before,
the series $\sum_{n = 1}^\infty n \tilde{c}_n(t)$ converges uniformly on compact intervals $[0, T]$.
With this in mind, we observe that $\sum_{n = 1}^\infty n^\lambda \frac{\tilde{c}_n^2}{2}$ is differentiable with

\begin{align*}
\partial_t \sum_{n = 1}^\infty n^\lambda \frac{\tilde{c}_n^2}{2} & = \sum_{n = 1}^\infty n^\lambda \tilde{c}_n \big(\tilde{J}_{n-1} - \tilde{J}_n \big) 
= \sum_{n = 0}^\infty (n + 1)^\lambda \tilde{c}_{n + 1} \tilde{J}_n - \sum_{n = 1}^\infty n^\lambda \tilde{c}_n \tilde{J}_n \\
&= - \sum_{n = 0}^\infty (n^\lambda \tilde{c}_n - (n+1)^\lambda \tilde{c}_{n+1}) \tilde{J}_n = - \sum_{n = 0}^\infty \big \vert \tilde{J}_n \big\vert ^2.
\end{align*}

Furthermore it holds 
\begin{equation*}
\partial_t \sum_{n = 0}^{\infty} \frac{1}{2} \big\vert\tilde{J}_n \big\vert^2 =  - \sum_{n = 0}^{\infty} (n+1)^\lambda \big( \tilde{J}_{n+1} - \tilde{J}_n \big)^2 \leq 0,
\end{equation*}
implying convergence of $\sum_{n = 0}^\infty \big\vert \tilde{J}_n(t)  \big\vert^2$ to some nonnegative constant as $t \to \infty$. Since
\begin{equation*}
\sum_{n = 1}^\infty n^\lambda \frac{\tilde{c}_n(t)^2}{2} 
+ \int_{0}^{t} \sum_{n = 0}^\infty \left\vert \tilde{J}_n(s) \right\vert ^2 \text{ d}s 
= \sum_{n = 1}^\infty n^\lambda \frac{\tilde{c}_n(0)^2}{2} < \infty
\end{equation*}
for $t \geq 0$,
it follows that 
$\sum_{n = 0}^\infty \big\vert \tilde{J}_n  \big\vert^2 \to 0$
as $t \to \infty$. From this and
$\tilde{J}_0 = -\tilde{c}_1$ we conclude $\tilde{c}_1(t) \to 0$ as $t \to \infty$, and
we obtain by induction that $\tilde{c}_n(t) \to 0$ for all $n \geq 1$, and by the comparision principle it follows that $c_n(t) \to 0$ for all $n \geq 1$.

{\bf Case II:  $f(0) \geq \frac{1}{\kappa}.$}
We make use of an idea from \cite{carr-dunwell}.
We consider the functions \linebreak
$q_k(t) = \sum_{n = k + 1}^\infty (n - k)c_n(t)$
for $k \geq 1$ and observe that
$q_k(t) \leq \sum_{n = k + 1}^\infty nc_n(t) \leq \rho,$
and 
\begin{align*}
\partial_t q_k &= \sum_{n = k + 1}^\infty (n-k)\left(J_{n-1} - J_n \right) 
= \sum_{n = k}^\infty (n + 1)J_n - \sum_{n = k + 1}^\infty n J_n - k\sum_{n = k + 1}^\infty \left(J_{n-1} - J_n \right).
\end{align*}
Since $\lim\limits_{n \to \infty} \left\vert J_n \right\vert \to 0$ uniformly on compact time intervals, it holds that $\sum_{n = k + 1}^\infty \left(J_{n-1} - J_n \right) = J_k,$
allowing us to further rewrite 
\begin{align*}
\partial_t q_k  &= \sum_{n = k}^\infty J_n 
= \kappa f\sum_{n = k}^\infty n^\lambda c_n - \sum_{n = k}^\infty (n+1)^\lambda c_{n+1}
= (\kappa f - 1) \sum_{n = k}^\infty n^\lambda c_n + k^\lambda c_k \geq 0,
\end{align*}
where we used that the assumption $f(0) \geq \frac{1}{\kappa }$ implies that $\kappa f(t) - 1 \geq 0$ for all $t \geq 0$. 
Thus each function $q_k$ grows in time. But since the $q_k$ are bounded from above, it must hold that $\partial_t q_k(t) \to 0$ as $t \to \infty$, which implies $k^\lambda c_k(t) \to 0$ as $t \to \infty$ for all $k \geq 1$.

In total, we have shown that for any $\lambda \in [0,1]$ and $f(0) \in [0,\rho]$ we have $c_n(t) \to 0$ as $t \to \infty$ for all $n \geq 1$.

\subsection{Clusters on the membrane: proof of Prop. \ref{P.longtime}, part \ref{P.longtimeb}}\label{subsec-monomers}
It remains to examine the long-time behaviour of $f$, the concentration of monomers in the cytosol. We immediately observe that if $f(0) = \frac{1}{\kappa}$, then $f(t) = \frac{1}{\kappa}$ for all $t \geq 0$, and when $f(0) = \rho$ it holds that $f(t) = \rho$ for all $t \geq 0$. For initial data other than these we consider two cases.

{\bf Case I: $\lambda = 1.$}
In this case, the equation for $f$ becomes
$$\partial_t f =  (1- \kappa f)(\rho-f)$$
and we find by standard ODE arguments that $f(t) \to \min \{\rho, \frac{1}{\kappa} \}$ exponentially fast.

\bigskip
{\bf Case II:  $\lambda < 1.$}
The equation for $f$ now takes the form
$$\partial_t f = (1-\kappa f)\sum_{n = 1}^\infty n^{\lambda}c_n.$$

We first consider the case that $\rho \leq \frac{1}{\kappa}$.
Then since $f$ is monotonically increasing and bounded from above by $\rho$, it must converge to some $\bar{f}$ as $t \to \infty$. Suppose that $\bar{f} < \rho$. For all $t \geq 0$ we have
$(\rho - f(t))^{2} \geq (\rho - \bar{f})^{2} > 0$
and by Hölder's inequality it holds that
\begin{equation}\label{longtime-rho-leq-half-hoelder}
\big(\rho - f(t) \big)^{2} = \Big(\sum_{n = 1}^\infty nc_n(t) \Big)^2 \leq \sum_{n = 1}^\infty n^\lambda c_n(t) \sum_{n = 1}^\infty n^{2 - \lambda}c_n(t).
\end{equation}

We notice that $(n+1)^{2 - \lambda} - n^{2-\lambda} \leq d(\lambda)  n^{1 - \lambda}$ for some constant $d(\lambda)$,
such that 
we find that by dropping negative terms and using mass conservation, we can formally bound (and rigorously by approximation argument)
\begin{align*}
\partial_t \sum_{n = 1}^{\infty} n^{2-\lambda}c_n(t) 
&= \sum_{n = 1}^{\infty} n^{2-\lambda}\big(J_{n-1} -J_n \big)  = \sum_{n = 0}^{\infty} (n+1)^{2-\lambda}J_n -\sum_{n = 1}^{\infty} n^{2-\lambda}J_n \\
&= -c_1 + \sum_{n = 1}^{\infty}\left(\kappa fn^\lambda c_n - (n+1)^\lambda c_{n+1} \right) \big((n+1)^{2 - \lambda} -n^{2 - \lambda} \big) \\
&\leq  \kappa f\sum_{n = 1}^{\infty} n^\lambda c_n \big((n+1)^{2 - \lambda} -n^{2 - \lambda} \big)  \\
&\leq d(\lambda) \kappa \rho \sum_{n = 1}^{\infty} n^\lambda c_n n^{1 - \lambda} \leq \kappa d(\lambda)\rho^{2}
\end{align*}
and hence 	$\sum_{n = 1}^\infty n^{2-\lambda}c_n(t) \leq \sum_{n = 1}^\infty n^{2-\lambda}c_n(0) + \kappa \rho^{2}  d(\lambda) t.$
Combining this with \eqref{longtime-rho-leq-half-hoelder} yields
$$\sum_{n = 1}^\infty n^\lambda c_n(t) \gtrsim \left(\rho -  \bar{f} \right)^{2} \frac{1}{1+t}.$$
This is not integrable, so
$$f(t) = \frac{1}{\kappa} + \frac{\kappa f(0) - 1}{\kappa}e^{-\kappa \int_{0}^{t} \sum_{n = 1}^\infty n^\lambda c_n(s) \text{ d}s} \to \frac{1}{\kappa} \neq \bar{f} $$
as $t \to \infty$, a contradiction. Hence it must be that $\bar{f} = \rho$.

A similar approach also works in the case $\rho > \frac{1}{\kappa}$. If $f(0) < \frac{1}{\kappa}$ it holds  $\rho - f \geq \rho - \frac{1}{\kappa}$ for all times, such that $\sum_{n = 1}^\infty n^\lambda c_n(t) \gtrsim \big(\rho -  \tfrac{1}{\kappa} \big)^{2} \frac{1}{1+t}.$
On the other hand, if $f(0) > \frac{1}{\kappa}$ we have  $\big(\rho - f \big)^{2} \geq \left(\rho - f(0) \right)^{2}$. In both cases it follows that
$$f(t) = \frac{1}{\kappa} + \frac{\kappa f(0) - 1}{\kappa}e^{-\kappa \int_{0}^{t} \sum_{n = 1}^\infty n^\lambda c_n(s) \text{ d}s} \to \frac{1}{\kappa} \quad \mbox{ as } t \to \infty.$$

We remark that in the case $\rho > \frac{1}{\kappa}$ and $f(0) > \frac{1}{\kappa}$, we can conclude without the additional assumption  $\sum_{n = 1}^{\infty} n^{2 - \lambda}c_n(0) < \infty$ by bounding $\sum_{n=1}^{\infty} c_n(t)$ from 
below by the solution of the pure diffusion problem (\ref{pure-diffusion-system}). One can then use a result from \cite{eichenberg-schlichting}   that characterizes the time decay of $\sum_{n=1}^{\infty} \tilde c_n(t)$.
We omit the details here.

\section{Proof of Theorem \ref{T.main}}\label{sec-selfsim}

In this section we prove the main result of this paper, that is the self-similar long-time behaviour for $a_n = b_n = n$ and $\rho>\frac{1}{\kappa}$. 
Our results regard the tails $C_k(t) = \sum_{n = k}^\infty c_n(t)$ for $k \geq 0$, where we set $c_0(t) \equiv 0$. We obtain
\begin{align*}
\partial_t C_k(t) &= \sum_{n = k}^\infty \partial_t c_n(t) = \sum_{n = k}^\infty (J_{n - 1} - J_n) = J_{k - 1} 
= \kappa f(k-1)c_{k-1} - kc_k \\
&= (k-1)\left(C_{k-1} - C_k \right) - k \left(C_k - C_{k+1}  \right) + (\kappa f - 1)(k-1)(C_{k-1} - C_k)
\end{align*}
for $k \geq 2$ and 
\begin{align*}
\partial_t C_1 &= J_0 = -(C_1 - C_2),
\end{align*}
for $k = 1$. The tails satisfy the discrete Neumann boundary condition
$C_1(t) - C_0(t) = 0$ for all $t \geq 0$.
We denote $\partial^+ C_k \coloneqq C_{k+1} - C_k$, $\partial^- C_k \coloneqq C_k - C_{k-1}$ and 
$$LC_k(t) = \partial^- \left(k \partial^+ C_k(t) \right) = (k-1)\left(C_{k-1}(t) - C_k(t) \right) - k \left(C_k(t) - C_{k+1}(t) \right).$$
We further denote
$$G_k(t) = (\kappa f(t) - 1)(k-1)c_{k-1},$$
such that the expression for $\partial_t C_k(t)$ becomes
\begin{equation}\label{selfsim-discrete-eqn-for-tails-simple-notation}
\partial_t C_k(t) = LC_k(t) + G_k(t).
\end{equation}

Taking the viewpoint that $G$ is an inhomogeneity and in light of the fact that $G_k(t) \to 0$ exponentially as $t \to \infty$, it seems intuitive that the solutions of \eqref{selfsim-discrete-eqn-for-tails-simple-notation} might behave for large times like the solutions of the homogeneous system
\begin{equation}\label{selfsim-discrete-pure-diffusion-eqn-for-tails}
\partial_t \tilde{C}_k(t) = L\tilde{C}_k(t)
\end{equation}
with boundary condition $k\partial^+C_k(t)\vert_{\{ k = 0\}} = 0$ for all $t \geq 0$.
This system can be viewed as a spatially discrete version of the system
\begin{equation}\label{selfsim-cont-pure-diffusion-eqn-for-tails}
\left\{ 
\begin{split}
&\partial_s \bar{C}(s,x) = \partial_x \left(x \partial_x \bar{C} \right), &&(s,x) \in (0, T) \times (0, \infty) \\
&\left(x\partial_x \bar{C} \right)\vert_{x = 0} = 0,  &&t \geq 0 \\
&\bar{C}(0, x) = C_0(x), &&x \geq 0
\end{split}
\right.
\end{equation}
for which dimensional analysis suggests the scaling $x \sim t$. Since equation \eqref{selfsim-cont-pure-diffusion-eqn-for-tails} conserves the zeroth moment of $\bar{C}(t,x)$, we make the Ansatz
$\bar{C}(t,x) = \frac{1}{t}h(\frac{x}{t})$. Inserting this into \eqref{selfsim-cont-pure-diffusion-eqn-for-tails} and using the boundary condition $\partial_x \bar{C}(t,x)\vert_{\{x = 0 \}} = 0$ for all $t \geq 0$ yields that $h(z) = de^{-z}$ for some constant $d$. 

Since
$\sum_{k=1}^\infty C_k(t) =  \sum_{n = 1}^\infty nc_n(t),$
we expect that the scaling solution should have as its zeroth moment the mass $\rho - \frac{1}{\kappa}$, which leads to $d = \rho - \frac{1}{\kappa}$. Therefore we arrive at
$$\bar{C}(t,x) = \frac{(\rho - \frac{1}{\kappa})}{t}e^{-\frac{x}{t}},$$
and we expect for $t \to \infty$ that
$$tC_{\lfloor tx \rfloor + 1}(t) \sim t\bar{C}(t, tx) = \Big(\rho - \frac{1}{\kappa}\Big)e^{-x}.$$

The key tool in the rigorous analysis will be to compare the rescaled tails to solutions of the spatially continuous homogeneous equation \eqref{selfsim-cont-pure-diffusion-eqn-for-tails}. Our approach is inspired by and makes use of the results in  \cite{eichenberg-schlichting}, who show a self-similarity result for solutions of the homogeneous equation \eqref{selfsim-discrete-pure-diffusion-eqn-for-tails}. Our situation is made more complicated by the presence of the inhomogeneity $G_k(t)$, since estimates for the fundamental solution of \eqref{selfsim-discrete-pure-diffusion-eqn-for-tails} as derived in \cite{eichenberg-schlichting} are degenerate at $t = 0$. However, expressing $C_k(t)$ in terms of this fundamental solution involves a convolution in time and space. Thus estimates for the fundamental solution do not translate to $C_k(t)$ due to their non-integrability. In particular, this prohibits us from establishing equicontinuity of suitable rescaled functions in time and space from the estimates on the fundamental solution. 

Our strategy of proof will be as follows: Given any sequence $t_j \to \infty$, we consider the functions $t_j C_{\lfloor t_j \cdot \rfloor + 1}(t_j)$ as a sequence of functions indexed by $t_j$. Since $t_j$ now takes the role of an index, we add a new time variable $s$ and denote
$$\Ctj(s, x) \coloneqq t_jC_{\lfloor t_j \cdot \rfloor + 1}(t_js).$$
We show that this sequence converges in the weak* sense in $L^\infty((a, T) \times (0, \infty))$ for any $0 < a < T$ and consists of approximate weak solutions of (\ref{selfsim-cont-pure-diffusion-eqn-for-tails}).
This will allow us to uniquely determine the weak* limit in  $L^\infty((a, T) \times (0, \infty))$. However, weak* convergence in $L^\infty((a, T) \times (0, \infty))$ does not necessarily hold for $s = 1$. To obtain a result for $s = 1$, we establish equicontinuity of the $\Ctj$ in the new time $s$ in a suitable weak sense, which will allow us to construct via the Arzel\`{a}-Ascoli theorem another limit of the $\Ctj$, and comparing limits allows us to identify it as $(\rho - \frac{1}{\kappa}) e^{-x}$.

\subsection{Weak* convergence in $L^\infty$}

In this section we prove equation \eqref{main1} in Theorem \ref{T.main}.
We first collect some estimates to use later.
\begin{lemma}\label{selfsim-estimates-lemma}
	\begin{enumerate}[(i)]
		\item $\sum_{k=1}^\infty G(t,k) \lesssim e^{-dt}$ for a constant $d > 0$ depending on $f(0)$.
		\item The unique solution $\phi(t,k,l)$ of (\ref{selfsim-discrete-pure-diffusion-eqn-for-tails}) with boundary condition $k\partial^+\phi(t, k, l)\vert_{\{ k = 0\}} = 0$ and initial condition
		$\phi(0, k, k) = 1$, $\phi(0, k, l) = 0$ for $k \neq l$ satisfies $\sup_{k \geq 1} \vert \phi(t,k,l) \vert \lesssim \frac{1}{1+t}.$
		\item For all $t \geq 0$ the tails $C_k(t)$ satisfy 
		\begin{equation}\label{selfsim-tail-decay}
		\sum_{k=1}^\infty C_k(t) \leq \rho, \quad \sup_{k \geq 1} \vert C_k(t) \vert \lesssim \frac{1}{1+t}.
		\end{equation}
		\item For all $s \in [0, \infty)$ the rescaled tails $\Ctj$ satisfy 
		\begin{equation}\label{selfsim-Ctj_L1}
		\int_{0}^{\infty} \Ctj(s,x) \text{ d}x \leq \rho
		\end{equation}
		and for all $s \in (0, \infty)$
		\begin{equation}\label{selfsim-Ctj_Linfty}
		\lVert \Ctj(s, \cdot) \rVert_{L^{\infty}((0, \infty))} \lesssim \frac{1}{s}.
		\end{equation}
		\item For any two measures $U \in \mathcal{M}(\N)$ and $\mu \in \mathcal{M}([0, \infty))$ it holds that
		\begin{equation}\label{selfsim-e&s-adjoint}
		\int_{[0, \infty)} tU(\floorxt) \text{ d}\mu(x) = t \sum_{k=1}^\infty U(k) \mu\Big(\Big[\frac{k-1}{t}, \frac{k}{t}\Big)\Big).
		\end{equation}
	\end{enumerate}
\end{lemma}

\begin{proof}
	Part $(i)$ follows from Section \ref{subsec-monomers}, while $(ii)$ is proved in 
	\cite[Prop. 2.4, Cor. 2.8, Lemma 2.14]{eichenberg-schlichting}.
	The first estimate in $(iii)$ is a consequence of $\sum_{n = 1}^\infty nc_n(t) \leq \rho$. As for the second, viewing $G_k(t)$ as an inhomogeneity, we can express the solutions $C_k(t)$ as
	\begin{equation}\label{self-sim-formula-C(t,k)-via-Phi}
	C_k(t) = \sum_{l = 1}^{\infty}\phi(t,k,l)C_{l}(0) + \int_{0}^{t} \sum_{l = 1}^\infty \phi(t-s,k,l)G_l(s) \text{ d}s.
	\end{equation}
	The bound then follows from $(i)$ and $(ii)$.
	Part $(iv)$ is a consequence of the first estimate in $(iii)$ and $(v)$ just follows from the definition.
\end{proof}

Next, we define an appropriate notion of weak solutions of \eqref{selfsim-cont-pure-diffusion-eqn-for-tails} as in \cite[Def. 3.4]{eichenberg-schlichting}. We use a particular class of test functions chosen such that uniqueness of weak solutions as distributions follows immediately. They are defined via the adjoint equation
\begin{equation}\label{selfsim-continuous-pure-diffusion-eqn-adjoint}
\left\{ 
\begin{split}
&\partial_s \varphi - \partial_x \left(x \partial_x \varphi \right) = g &&(s,x) \in (0, \infty) \times (0, \infty) \\
&\left(x\partial_x \varphi \right)\vert_{x = 0} = 0,  &&t \geq 0 \\
&\varphi(0, x) = 0, &&x \geq 0.
\end{split}
\right.
\end{equation}
Let $\psi(s,x,y)$ denote the fundamental solution of (\ref{selfsim-cont-pure-diffusion-eqn-for-tails}) derived in \cite[Section 3.1]{eichenberg-schlichting}. Then for any $g \in C_c^\infty((0, T) \times (0, \infty))$ the function
$$\xi(s,x) = \int_{0}^{s} \int_{0}^{\infty} \psi(s-r, x, y)g(r,y) \text{ d}y \text{ d}r$$
solves (\ref{selfsim-continuous-pure-diffusion-eqn-adjoint}).
As test functions we will use for given $T > 0$
\begin{equation}\label{selfsim-continuous-pure-diffusion-eqn-test-fcts}
\varphi(s,x) = \xi(T - s, x)g(x) = \int_{0}^{T - s} \int_{0}^{\infty} \psi(T -s-r, x, y)g(r,y) \text{ d}y \text{ d}r
\end{equation}
for $g \in C_c^\infty((0, T) \times (0, \infty))$. 
Following \cite{eichenberg-schlichting}, we define weak solutions as follows.
\begin{definition}\label{selfsim-def-weak-sol}
	A weak solution of (\ref{selfsim-cont-pure-diffusion-eqn-for-tails}) with initial data $\mu_0 \in \mathcal{M}([0, \infty))$ is a family \linebreak $\left\{ \mu_s \right\}_{s \in (0,T] } \subset \mathcal{M}([0, \infty))$ such that for all functions $\varphi$ of the form (\ref{selfsim-continuous-pure-diffusion-eqn-test-fcts}) it holds that
	\begin{equation}\label{selfsim-continuous-pure-diffusion-weak}
	\int_{0}^{T} \int_{[0, \infty)} \partial_s \varphi + \partial_x \left(x \partial_x \varphi \right) \text{d}\mu_s(x) \text{ d}s = - \int_{[0, \infty)} \varphi(0, x) \text{ d}\mu_0(x).
	\end{equation}
\end{definition}
\begin{remark}
	Weak solutions are unique in the sense of distributions, which can be seen by rewriting the above formulation to
	\begin{equation*}
	\int_{0}^{T} \int_{[0, \infty)} g(T-s,x) \text{ d}\mu_s(x) \text{ d}s = - \int_{[0, \infty)} \int_{0}^{T} \int_{0}^{\infty} \psi(T-r, x, y)g(r, y) \text{ d}y \text{ d}r \text{ d}\mu_0(x),
	\end{equation*}
	for $g \in C_c^\infty((0, T) \times (0, \infty))$. This will later allow us to identify the limit of $\Ctj$ as $j \to \infty$.
\end{remark}
The fact that test functions $\varphi$ are bounded and continuous on $[0,T] \times [0, \infty)$ and differentiable on $(0, T) \times (0, \infty)$ follows from the formula \eqref{selfsim-continuous-pure-diffusion-eqn-test-fcts} and the properties of $\psi$ shown in \cite[Prop. 3.2]{eichenberg-schlichting}. Furthermore, the authors show that
\begin{equation}\label{selfsim-varphi-est-1}
\sup_{x \in (0, \infty)}  \left\vert \partial_x \left( x  \partial_x  \varphi(s, x) \right)  \right\vert
\leq T \sup_{x \in (0, \infty), s \geq 0} \left\vert \partial_x \left( x  \partial_x  g(s, x) \right) \right\vert
< \infty
\end{equation}
for all $s \geq 0$, 
\begin{equation}\label{selfsim-varphi-est-2}
\sup_{x \in (0, \infty)} \vert\partial_x \varphi(s, x)\vert 
\leq \sup_{x \in (0, \infty)}  \partial_x \left( x  \partial_x  \varphi(s, x) \right), 
\end{equation}
and
\begin{equation}\label{selfsim-varphi-est-3}
\sup_{x \in (0, \infty)} \vert\partial_x^2 \varphi(s, x)\vert 
\leq 2x^{-1}\sup_{x \in (0, \infty)}  \partial_x \left( x  \partial_x  \varphi(s, x) \right), 
\end{equation}
see  \cite[Cor. 3.3, Lemma 3.9]{eichenberg-schlichting}, which imply
\begin{equation}\label{selfsim-varphi-est-4}
\sup_{x \in (0, \infty)} \vert x\partial_x^3 \varphi(s, x)\vert 
\lesssim x^{-1} + x^{-2}.
\end{equation}
With these estimates we are now ready for our first step, which is to determine the weak* limit of $\Ctj$.
\begin{theorem}\label{thm-weak*-limit_sx}
	There exists a subsequence of $t_j$, which we again denote by $t_j$, such that $$\Ctj(s,x) \weakstar \frac{(\rho - \frac{1}{\kappa})}{s}e^{-\frac{x}{s}}$$ in $L^\infty([a, \infty) \times (0, \infty) )$ as $j \to \infty$ for all $a > 0$.
\end{theorem}
\begin{proof} 
	By \eqref{selfsim-Ctj_Linfty} and the Banach-Alaoglu theorem, for any $a > 0$ we can extract from $t_j$ a subsequence (not relabelled) such that $\Ctj \weakstar \bar{C}$ in $L^\infty([a, \infty) \times (0, \infty))$ for some \linebreak $\bar{C} \in L^\infty([a, \infty) \times (0, \infty))$. Thus, by a diagonal scheme we obtain a subsequence and a function $\bar{C} \in  L^\infty((0, \infty) \times (0, \infty))$ such that $\Ctj \weakstar \bar{C}$ in $L^\infty([a, \infty) \times (0, \infty))$ for all $a > 0$.
	
	We now want to identify the limit $\bar{C}$. We will first show that it is a weak solution according to Definition \ref{selfsim-def-weak-sol}.
	
	{\bf Step 1:}
	Let $\varphi$ be a function of the type (\ref{selfsim-continuous-pure-diffusion-eqn-test-fcts}) for some $g \in C_c^\infty$. Then we have for any $j \geq 1$ that
	\begin{align*}
	\int_{0}^{T} &\int_{[0, \infty)} \Ctj (s,x) \partial_s \varphi(s,x) \text{ d}x \text{ d}s = \int_{0}^{T} \int_{[0, \infty)} t_jC_{\lfloor t_jx \rfloor + 1}(t_js ) \partial_s \varphi(s,x) \text{ d}x \text{ d}s \\
	&= t_j \int_{0}^{T } \sum_{k=1}^\infty C_k(t_js) \int_{(k-1)t_{j}^{-1}}^{kt_j^{-1}} \partial_s \varphi(s, x) \text{ d}x \text{ d}s \\
	&= -t_j \int_{0}^{T } \sum_{k=1}^\infty \partial_sC_k(t_js)  \piphi \text{ d}s 
	+ t_j  \sum_{k=1}^\infty C_k(t_jT) \int_{(k-1)t_{j}^{-1}}^{kt_j^{-1}} \varphi(T, x) \text{ d}x  \\
	&\quad\,- t_j  \sum_{k=1}^\infty C_k(0) \int_{(k-1)t_{j}^{-1}}^{kt_j^{-1}}  \varphi(0, x) \text{ d}x, \\
	\end{align*}
	where we have used (\ref{selfsim-e&s-adjoint}) with $U(k) = C_k(t_js)$ and integration by parts. Using the equation for $C$ and that $\varphi(T,x) = 0$ for all $x \geq 0$, we further compute
	\begin{align*}
	\int_{0}^{T} \int_{[0, \infty)} \Ctj (s,x) \partial_s \varphi(s,x) \text{ d}x \text{ d}s	&= -t_j^2 \int_{0}^{T } \sum_{k=1}^\infty LC_k(t_js)  \piphi \text{ d}s \\
	&\quad \,- t_j^2 \int_{0}^{T } \sum_{k=1}^\infty G_k(t_js)  \int_{(k-1)t_{j}^{-1}}^{kt_j^{-1}} \varphi(s, x) \text{ d}x \text{ d}s \\
	&\quad\,- t_j  \sum_{k=1}^\infty C_k(0) \int_{(k-1)t_{j}^{-1}}^{kt_j^{-1}} \varphi(0, x) \text{ d}x \\
	&\eqqcolon I + II + III.
	\end{align*}
	
	The operator $L$ is symmetric. Thus we can write
	\begin{align*}
	I &= -t_j^2 \int_{0}^{T } \sum_{k=1}^\infty C_k(t_js) L\left( \piphi \right) \text{ d}s \\
	&= -t_j \int_{0}^{T } \sum_{k=1}^\infty C_k(t_js)  \int_{(k-1)t_{j}^{-1}}^{kt_j^{-1}} \partial_x \left( x  \partial_x  \varphi(s, x) \right) \text{ d}x  \text{ d}s \\
	&\quad \,- t_j \int_{0}^{T } \sum_{k=1}^\infty C_k(t_js) \underbrace{ \left[ t_jL\left( \piphi \right) -  \int_{(k-1)t_{j}^{-1}}^{kt_j^{-1}} \partial_x \left( x  \partial_x  \varphi(s, x) \right) \text{ d}x \right] }_{=: \mathcal{R}_j(\varphi, k, s)} \text{ d}s \\
	&= - \int_{0}^{T} \int_{[0, \infty)} t_jC_{\lfloor t_jx \rfloor + 1}(t_js) \partial_x \left( x  \partial_x  \varphi(s, x) \right) \text{ d}x  \text{ d}s 
	- t_j \int_{0}^{T } \sum_{k=1}^\infty C_k(t_js) \mathcal{R}_j(\varphi, k, s) \text{ d}s,
	\end{align*}
	where we have used \eqref{selfsim-e&s-adjoint}. We now have
	\begin{align}\label{selfsim-approx-eqn-Ctj}
	\int_{0}^{T} &\int_{[0, \infty)} \Ctj (s,x) \Big( \partial_s \varphi(s,x) + \partial_x \left( x  \partial_x  \varphi(s, x) \right) \Big) \text{ d}x \text{ d}s \nonumber \\ 
	&= - t_j \int_{0}^{T } \sum_{k=1}^\infty C_k(t_js) \mathcal{R}_j(\varphi, k, s) \text{ d}s + II + III.
	\end{align}
	Since $\partial_s \varphi + \partial_x \left( x  \partial_x  \varphi \right) = g \in C_c^\infty((0,T) \times (0, \infty) )$ and $\Ctj \rightharpoonup^* \bar{C}$ in $L^\infty([a,\infty) \times (0, \infty) )$ for all $a > 0$, the left hand side converges to
	$$\int_{0}^{T} \int_{[0, \infty)} \bar{C} (s,x) \Big( \partial_s \varphi(s,x) + \partial_x \left( x  \partial_x  \varphi(s, x) \right) \Big) \text{ d}x \text{ d}s$$
	as $j \to \infty$.
	
	{\bf Step 2:}
	We want to understand the limits of the terms on the right hand side of \ref{selfsim-approx-eqn-Ctj} as $j \to \infty$.
	We begin with $III$.
	First, we observe that by (\ref{selfsim-e&s-adjoint}),
	\begin{align*}
	\int_{[0, \infty)} \Ctj(0,x) \text{ d}x = t_j \int_{[0, \infty)} C_{\lfloor t_jx \rfloor + 1}(0) \text{ d}x &= \sum_{k=1}^\infty C_k(0) = \rho - f(0).
	\end{align*}
	
	Because $\varphi(0,\cdot)$ is continuous, for $\varepsilon > 0$ there exists some  $\delta > 0$ such that $0 \leq x < \delta$ implies $\vert \varphi(0,x) - \varphi(0,0) \vert < \varepsilon$. Furthermore, $\varphi$ is bounded and hence there exists a constant $M > 0$ such that $\left\vert \varphi(0,x) \right\vert \leq M$ for all $x \geq 0$.
	We further notice that for any $\varepsilon > 0$ and $\delta > 0$ there exists $j$ sufficiently large such that
	\begin{align*}
	\int_{[\delta, \infty)} \Ctj(0,x) \text{ d}x &= t_j \int_{[\delta, \infty)} C_{\lfloor t_jx \rfloor + 1}(0) \text{ d}x \leq \sum_{k \geq \lfloor \delta t_j \rfloor + 1} C_k(0) \leq \varepsilon
	\end{align*}
	since $\sum_{k=1}^\infty kc_k(0) < \infty$. Hence
	\begin{align*}
	\Big\vert  \int_{[0, \infty)}& \Ctj(0,x) \varphi(0,x) \text{ d}x - (\rho - f(0)) \varphi(0,0) \Big\vert \\
	&= \Big\vert  \int_{[0, \infty)} \Ctj(0,x) \left(\varphi(0,x) - \varphi(0,0) \right)  \text{ d}x \Big\vert \\
	&\leq \int_{[0, \delta)} \Ctj(0,x) \left\vert\varphi(0,x) - \varphi(0,0) \right\vert  \text{ d}x \,+ \int_{[\delta, \infty)} \Ctj(0,x) \left\vert\varphi(0,x) - \varphi(0,0) \right\vert  \text{ d}x \\
	&\leq \left(\rho - f(0)\right) \varepsilon + 2M \varepsilon,
	\end{align*}
	and therefore
	\begin{equation}\label{selfsim-limit-eqn-initial-value-1}
	III = -\int_{[0, \infty)} \Ctj(0,x) \varphi(0,x) \text{ d}x \to (f(0) - \rho) \varphi(0,0)
	\end{equation}
	as $j \to \infty$.
	
	{\bf Step 3:}
	We now turn our eye to II, the term involving the inhomogeneity $G$. We observe that by \eqref{selfsim-e&s-adjoint},
	\begin{align*}
	t_j^2 \int_{0}^{T } &\int_{0}^{\infty}  G_{\lfloor t_j x \rfloor + 1}(t_js) \text{ d}x \text{ d}s = t_j^2 \int_{0}^{T } \sum_{k=1}^\infty G_k(t_js)  \int_{(k-1)t_{j}^{-1}}^{kt_j^{-1}} 1 \text{ d}x \text{ ds} \\
	&= \int_{0}^{t_jT} (\kappa f(r) - 1)(\rho - f(r)) \text{ d}r 
	= - \int_{0}^{t_jT} \partial_r f(r) \text{ d}r = f(0) - f(t_jT).
	\end{align*}
	By \eqref{selfsim-e&s-adjoint}, the decay of $G$ from Lemma \ref{selfsim-estimates-lemma}, the monotonicity of $f$ and the continuity and boundedness of $\varphi$, similarly to the previous computation it follows that
	\begin{equation}\label{selfsim-convergence-G-term-to-dirac}
	II = -t_j^2 \int_{0}^{T } \sum_{k=1}^\infty G_k(t_js)  \int_{(k-1)t_{j}^{-1}}^{kt_j^{-1}} \varphi(s, x) \text{ d}x \text{ ds} \to \Big(\frac{1}{\kappa} - f(0) \Big) \varphi(0,0)
	\end{equation}
	as $j \to \infty$.
	In combination with (\ref{selfsim-limit-eqn-initial-value-1}), this gives
	\begin{align*}
	II + III \To  -\Big(\rho - \frac{1}{\kappa}\Big)  \int_{[0, \infty)} \varphi(0, x) \text{ d}\delta_0(x).
	\end{align*}
	
	{\bf Step 4:} The fact that
	\begin{equation*}
	t_j \int_{0}^{T } \sum_{k=1}^\infty C_k(t_js) \mathcal{R}_j(\varphi,k, s) \text{ d}s \to 0
	\end{equation*}
	as $j \to \infty$, where
	\begin{equation*}
	\mathcal{R}_j(\varphi,k, s) = t_jL\Big( \piphi \Big) -  \int_{(k-1)t_{j}^{-1}}^{kt_j^{-1}} \partial_x \left( x  \partial_x  \varphi(s, x) \right) \text{ d}x,
	\end{equation*}
	is shown in \cite[Prop. 3.7]{eichenberg-schlichting}, who treat the homogeneous problem using the same class of test functions. 
	
	{\bf Conclusion:} We can therefore take the limit in \eqref{selfsim-approx-eqn-Ctj}
	to obtain that
	\begin{align*}
	\int_{0}^{T} \int_{[0, \infty)} \left( \partial_s \varphi + \partial_x \left(x \partial_x \varphi \right) \right) \bar{C}(s,x) \text{ d}x \text{ d}s = -\Big(\rho - \frac{1}{\kappa}\Big) \int_{[0, \infty)} \varphi(0, x) \text{ d}\delta_0(x),
	\end{align*}
	that is, $\bar{C}$ is a weak solution of \eqref{selfsim-cont-pure-diffusion-eqn-for-tails} with initial data $\mu_0 = (\rho - \frac{1}{\kappa}) \delta_0$. One easily checks that $\frac{(\rho - \frac{1}{\kappa})}{s}e^{-\frac{x}{s}}$ is also a weak solution with the same initial data, and since weak solutions are distributionally unique by \eqref{selfsim-continuous-pure-diffusion-weak} it follows that $\bar{C}(s,x) = \frac{(\rho - \frac{1}{\kappa})}{s}e^{-\frac{x}{s}}$ almost everywhere in $(0, T) \times (0, \infty)$, concluding the proof.
\end{proof}
Our primary interest though is in the behaviour of $\Ctj$ for $s = 1$, and this is a null set in $(0, T) \times (0, \infty)$. Our strategy will be to consider another limit of $\Ctj$ that is attained uniformly in time, and to deduce by comparison that this limit must be equal almost everywhere to $\bar{C}$. The difficulty here is establishing equicontinuity in time for $\{\Ctj \}$ in a suitable sense. This is achieved in the next section.

\subsection{Weak equicontinuity in time}
We will use a variation of the Arzel\`{a}-Ascoli theorem that requires equicontinuity only in a weak sense \cite[A.3]{vrabie}.

Let $Y$ be a real Banach space and $a < b$ be real numbers. Then the set of weakly continuous functions $u: [a,b] \to Y$ is given by
\begin{equation*}
C([a,b], Y_w) \coloneqq \big\{  u: [a,b] \to Y \text{ s.t. } \forall y^* \in Y^* : t \mapsto \langle y^*, u(t) \rangle \in C([a,b]) \big\},
\end{equation*}
where $\langle \cdot, \cdot \rangle$ denotes the dual pairing.

A topology on $C([a,b], Y_w)$ can be defined via the family of seminorms $ \lVert \cdot \rVert_{\mathcal{Y}^{*}}$ for nonempty and finite
$\mathcal{Y}^* \subset Y^*$ given by
\begin{equation*}
\lVert u \rVert_{\mathcal{Y}^{*}} \coloneqq \sup_{t \in [a,b]} \sup_{y^{*} \in \mathcal{Y}^{*}} \vert \langle y^*, u(t) \rangle \vert.
\end{equation*}
A variation of the classical Arzel\`{a}-Ascoli theorem characterizes sequentially relatively compact sets in this space. 
\begin{theorem}\label{selfsim-weak-arzela-ascoli}
	Let the real Banach space $Y$ be weakly sequentially complete. A family \linebreak $\mathcal{F} \subset C([a,b], Y_w)$ is sequentially relatively compact if and only if
	\begin{enumerate}[(i)]
		\item $\mathcal{F}$ is weakly equicontinuous, that is, for all $y^* \in Y^*$, the family $\left\{t \mapsto \langle y^*, u(t) \rangle  \right\}_{u \in \mathcal{F}}$ is equicontinuous on $[a,b]$, and
		\item there exists a dense subset $D \subset [a,b]$ such that for every $t \in D$ the set  $\left\{u(t): u \in \mathcal{F} \right\} \subset Y$ is sequentially weakly relatively compact.
	\end{enumerate}
\end{theorem}

As the target space $Y$ we choose the dual of an appropriately weighted Sobolev space.

\begin{definition}
	We define the weighted Sobolev space
	\begin{align*}
	&\quad \; \Sobolev \\
	&\coloneqq \Big\{u \in W^{2,2}((0, \infty)): 
	\Vert u\Vert_W:=\int_{0}^{\infty} (u^2 + (\partial_x u)^2 + (\partial_x^2 u)^2)(1 + x)^2 \text{ d}x < \infty  \Big\}.
	\end{align*}
	and denote for brevity $W := \Sobolev$. 
\end{definition}

The space $W$ endowed with $\Vert\cdot \Vert_W$ is a reflexive Banach space.
We identify $\Ctj(s, \cdot)$ with the function in $W^*$ given by
\begin{equation*}
u \mapsto \int_{0}^{\infty} \Ctj(s,x)u(x) \text{ d}x.
\end{equation*}
Indeed this function is linear and continuous on $W$ for $s > 0$ since by \eqref{selfsim-Ctj_L1} and \eqref{selfsim-Ctj_Linfty}
\begin{align*}
\int_{0}^{\infty} \Ctj(s,x)u(x) \text{ d}x 
&\leq \left(\int_{0}^{\infty}  u(x)^2 (1 + x)^2 \text{ d}x  \right)^{\frac{1}{2}} \left(\int_{0}^{\infty} \left( \Ctj(s,x) \right)^2  \text{ d}x \right)^{\frac{1}{2}} \\
&\leq \lVert u \rVert_{W} \lVert \Ctj(s, \cdot) \rVert_{L^{\infty}((0,\infty))}^{\frac{1}{2}} \lVert \Ctj(s, \cdot) \rVert_{L^{1}((0,\infty))}^{\frac{1}{2}} \\
&\lesssim \rho^{\frac{1}{2}}s^{-\frac{1}{2}}\lVert u \rVert_{W}.
\end{align*}

It is convenient to choose the dual of a reflexive space as the target space $Y$ since we can estimate $\langle u^{**}, \Ctj(s, \cdot) \rangle$ for $u^{**} \in W^{**}$ by $\langle \Ctj(s, \cdot), u \rangle$ for $u \in W$, for which we have an explicit formula.
Furthermore, since $W$ is reflexive it holds that $W^*$ is also reflexive and in particular weakly sequentially complete. 
We obtain the following result for $s = 1$. 

\begin{theorem}\label{selfsim-main theorem}
	There exists a subsequence $j \to \infty$ such that for all $u \in W$ it holds that
	$$\int_{0}^{\infty} \Ctj(1,x)u(x) \text{ d}x \to \int_{0}^{\infty} u(x) \Big(\rho - \frac{1}{\kappa}\Big) e^{-x} \text{ dx}$$
	as $j \to \infty.$
\end{theorem}

\begin{proof}
	Our first step is to show that $\Ctj$ has a subsequence that converges in $C([a,T], W^*_w)$ for all $a > 0$ using Theorem \ref{selfsim-weak-arzela-ascoli}, where $T > 1$ is arbitrary. Let $a > 0$.
	We prove that $\Ctj$ are weakly equicontinuous, which also implies that $\Ctj \in C([a,T], W^*_w)$ for all $j$. For this, we show that for all $u \in W$ the functions
	$\left\{ s \mapsto \langle  \Ctj(s, \cdot), u \rangle \right\}_{j \geq 1}$
	are equicontinuous on $[a,T]$, where 
	$$\langle  \Ctj(s, \cdot), u \rangle =  \int_{0}^{\infty} \Ctj(s,x)u(x) \text{ d}x.$$
	We compute
	\begin{align*}
	\partial_s  \int_{0}^{\infty} \Ctj(s,x)u(x) \text{ d}x 
	&= \partial_s \int_{0}^{\infty} t_j C_{\lfloor t_jx \rfloor + 1}(t_js)u(x) \text{ d}x
	= \partial_s \sum_{k=1}^\infty t_j C_k(t_js) \int_{(k-1)t_j^{-1}}^{kt_j^{-1}} u(x) \text{ d}x \\
	&= \sum_{k=1}^\infty t_j^2 LC_k(t_js) \int_{(k-1)t_j^{-1}}^{kt_j^{-1}} u(x) \text{ d}x 
	+ \sum_{k=1}^\infty t_j^2 G_k(t_js) \int_{(k-1)t_j^{-1}}^{kt_j^{-1}} u(x) \text{ d}x \\
	&\eqqcolon I + II.
	\end{align*}
	For the second summand, we have that
	\begin{align*}
	II &= t_j^2 \left(\kappa f(t_js) - 1 \right) \sum_{k=1}^\infty (k-1)c_{k-1} \int_{(k-1)t_j^{-1}}^{kt_j^{-1}} u(x) \text{ d}x \\
	&\leq t_j^2 \left(\kappa f(t_js) - 1 \right) \sup_{k \geq 1} \left((k-1)c_{k-1}\right)\sum_{k = 1}^\infty \int_{(k-1)t_j^{-1}}^{kt_j^{-1}} \vert u(x)\vert \text{ d}x,
	\end{align*}
	and by the exponential convergence of $\vert\kappa f - 1\vert \to 0$ and the fact that $e^{-x} \leq x^{-2}$ for all $x > 0$, this is bounded by
	\begin{align*}
	t_j^2 e^{-dt_{j}s} \int_{0}^{\infty} \vert u(x)\vert \text{ d}x 
	&\lesssim \frac{1}{s^2} \int_{0}^{\infty} \vert u(x)\vert \text{ d}x 
	\leq \frac{1}{a^2} \lVert u \rVert_{W}.
	\end{align*}
	We observe that by Morrey's theorem, 
	$W \subset W^{2,2}((0, \infty)) \subset C^1((0, \infty)).$
	So $u \in W$ is continuously differentiable, and for $k \geq 3$ it holds by definition of $L$ that
	\begin{align*}
	&\quad \, L\Big(\int_{(k-1)t_j^{-1}}^{kt_j^{-1}} u(x) \text{ d}x \Big) \\
	&= k \int_{(k-1)t_j^{-1}}^{kt_j^{-1}} u(x + t_j^{-1}) - u(x) \text{ d}x - (k-1) \int_{(k-2)t_j^{-1}}^{(k-1)t_j^{-1}} u(x + t_j^{-1}) - u(x) \text{ d}x \\
	&= k \int_{(k-1)t_j^{-1}}^{kt_j^{-1}} \int_{x}^{x + t_j^{-1}} \partial_s u(s) \text{ d}s \text{ d}x
	- (k-1) \int_{(k-2)t_j^{-1}}^{(k-1)t_j^{-1}} \int_{x}^{x + t_j^{-1}} \partial_s u(s) \text{ d}s \text{ d}x \\
	&= kt_j^{-1} \int_{(k-1)t_j^{-1}}^{(k+1)t_j^{-1}} \partial_s u(s) \text{ d}s 
	- (k-1)t_j^{-1} \int_{(k-2)t_j^{-1}}^{kt_j^{-1}}  \partial_s u(s) \text{ d}s \\
	&= kt_j^{-1} \int_{(k-1)t_j^{-1}}^{(k+1)t_j^{-1}} \partial_s u(s) - \partial_s u(s  - t_j^{-1}) \text{ d}s  + t_j^{-1} \int_{(k-2)t_j^{-1}}^{kt_j^{-1}} \partial_s u(s) \text{ d}s.
	\end{align*}
	By applying Hölder's inequality to the first term, we can estimate this by
	\begin{align*}
	&\quad \; kt_j^{-\frac{3}{2}} \Big(\int_{(k-1)t_j^{-1}}^{(k+1)t_j^{-1}} \Big(\partial_s u(s) - \partial_s u(s  - t_j^{-1}) \Big)^2 \text{ d}s   \Big)^{\frac{1}{2}}  
	+ t_j^{-1} \int_{(k-2)t_j^{-1}}^{kt_j^{-1}} \partial_s u(s) \text{ d}s \\
	&\leq kt_j^{-\frac{5}{2}} \Big( \int_{(k-1)t_j^{-1}}^{(k+1)t_j^{-1}} \Big( \frac{\partial_s u(s) - \partial_s u(s  - t_j^{-1})}{t_j^{-1}}\Big)^2  \text{ d}s \Big)^{\frac{1}{2}} 
	+ t_j^{-1} \int_{(k-2)t_j^{-1}}^{kt_j^{-1}} \partial_s u(s) \text{ d}s \\
	&\leq kt_j^{-\frac{5}{2}} \Big( \int_{(k-2)t_j^{-1}}^{(k+1)t_j^{-1}} (\partial_s^2 u(s))^2  \text{ d}s \Big)^{\frac{1}{2}} 
	+ t_j^{-1} \int_{(k-2)t_j^{-1}}^{kt_j^{-1}} \partial_s u(s) \text{ d}s,
	\end{align*}
	where we used that the Sobolev function $\partial_s u \in W^{1,2}((0, \infty))$ can be approximated by smooth functions on $((k-2)t_j^{-1}, (k+1)t_j^{-1})$, and for smooth $u$ we have by Jensen's inequality that
	\begin{align*}
	&\quad\; \Big( \int_{(k-1)t_j^{-1}}^{(k+1)t_j^{-1}} \Big( \frac{\partial_s u(s) - \partial_s u(s  - t_j^{-1})}{t_j^{-1}}\Big)^2  \text{ d}s \Big)^{\frac{1}{2}} \\
	&= \Big( \int_{(k-1)t_j^{-1}}^{(k+1)t_j^{-1}}  \Big(t_j \int_{s - t_j^{-1}}^{s} \partial_r^2 u(r) \text{ d}r\Big)^2  \text{ d}s \Big)^{\frac{1}{2}}  
	\leq \Big( \int_{(k-1)t_j^{-1}}^{(k+1)t_j^{-1}} t_j \int_{s - t_j^{-1}}^{s} \left(\partial_r^2 u(r) \right)^2 \text{d}r  \text{ d}s \Big)^{\frac{1}{2}} \\
	&= \Big( \int_{(k-2)t_j^{-1}}^{(k+1)t_j^{-1}} \left(\partial_r^2 u(r) \right)^2 \int_{r}^{r + t_j^{-1}} t_j \text{ d}s  \text{ d}r \Big)^{\frac{1}{2}}
	= \Big( \int_{(k-2)t_j^{-1}}^{(k+1)t_j^{-1}} (\partial_s^2 u(s))^2  \text{ d}s \Big)^{\frac{1}{2}}.
	\end{align*}
	Additionally, since $kt_j^{-1} \leq s + t_j^{-1}$ for all $s \geq (k-1)t_j^{-1}$ we have that
	\begin{align*}
	kt_j^{-\frac{5}{2}} \Big( \int_{(k-2)t_j^{-1}}^{(k+1)t_j^{-1}} (\partial_s^2 u)^2  \text{ d}s \Big)^{\frac{1}{2}} 
	\leq t_j^{-\frac{3}{2}} \Big( \int_{(k-2)t_j^{-1}}^{(k+1)t_j^{-1}} (s + 2t_j^{-1})^2 (\partial_s^2 u)^2  \text{ d}s \Big)^{\frac{1}{2}} 
	\end{align*}
	All in all, we have for $k \geq 3$
	\begin{align*}
	\Big\vert L\Big(\int_{(k-1)t_j^{-1}}^{kt_j^{-1}} u(x) \text{ d}x \Big) \Big\vert 
	\leq t_j^{-\frac{3}{2}} \Big( \int_{(k-2)t_j^{-1}}^{(k+1)t_j^{-1}} (s + 2t_j^{-1})^2 (\partial_s^2 u)^2  \text{ d}s \Big)^{\frac{1}{2}} 
	+ t_j^{-1} \int_{(k-2)t_j^{-1}}^{kt_j^{-1}} \vert \partial_s u(s)\vert \text{ d}s.
	\end{align*}
	For $k = 1$ we estimate
	\begin{align*}
	L\Big(\int_{(k-1)t_j^{-1}}^{kt_j^{-1}} u(x) \text{ d}x \Big)\Big\vert_{k = 1}
	&= \int_{0}^{t_j^{-1}} u(x + t_j^{-1}) - u(x) \text{ d}x \\
	&= \int_{0}^{t_j^{-2}} u(x + t_j^{-1}) - u(x) \text{ d}x
	+ \int_{t_j^{-2}}^{t_j^{-1}} u(x + t_j^{-1}) - u(x) \text{ d}x \\
	&\leq 2 \int_{0}^{t_j^{-2}} \vert u(x)\vert \text{ d}x 
	+ \int_{t_j^{-2}}^{t_j^{-1}} \int_{x}^{x + t_j^{-1}} \partial_s u(s) \text{ d}s \text{ d}x \\
	&\lesssim \Big(\int_{0}^{t_j^{-2}} \vert u(x)\vert^2 \text{ d}x \Big)^{\frac{1}{2}} t_j^{-1}
	+ t_j^{-1}\int_{t_j^{-2}}^{2t_j^{-1}} \partial_s u(s) \text{ d}s \\
	&\leq t_j^{-1} \left( \lVert u \rVert_{W} + \lVert u \rVert_{W}^2 \right),
	\end{align*}
	and for $k = 2$ we have by a similar computation that
	\begin{align*}
	L\Big(\int_{(k-1)t_j^{-1}}^{kt_j^{-1}} u(x) \text{ d}x \Big)\Big\vert_{k = 2}&= 2 \int_{t_j^{-1}}^{2t_j^{-1}} u(x + t_j^{-1}) \text{ d}x - \int_{0}^{t_j^{-1}} u(x + t_j^{-1}) \text{ d}x \\
	&\lesssim t_j^{-1} \left( \lVert u \rVert_{W} + \lVert u \rVert_{W}^2 \right).
	\end{align*}
	Inserting these estimates into $I$ and using the symmetry of $L$ yields
	\begin{align*}
	I 
	&= t_j^2 \sum_{k=1}^\infty C_k(t_js) L\Big(\int_{(k-1)t_j^{-1}}^{kt_j^{-1}} u(x) \text{ d}x \Big) \\
	& \lesssim t_j^{\frac{1}{2}} \sum_{k=3}^\infty C_k(t_js) \Big( \int_{(k-2)t_j^{-1}}^{kt_j^{-1}} (s + 2t_j^{-1})^2 (\partial_s^2 u)^2  \text{ d}s \Big)^{\frac{1}{2}} \\
	&\quad \, + t_j \sum_{k=3}^\infty C_k(t_js) \int_{(k-3)t_j^{-1}}^{(k-1)t_j^{-1}} \partial_s u(s) \text{ d}s \\
	&\quad \, + t_j^2 \left( C_1(t_js) + C_2(t_js) \right) t_j^{-1} \left( \lVert u \rVert_{W} + \lVert u \rVert_{W}^2 \right) \\
	&\lesssim t_j^{\frac{1}{2}} \Big( \sum_{k=3}^\infty C_k(t_js)^2 \Big)^{\frac{1}{2}} \Big( \sum_{k=3}^\infty  \int_{(k-2)t_j^{-1}}^{kt_j^{-1}} (s + 1)^2 (\partial_s^2 u(s))^2  \text{ d}s \Big)^{\frac{1}{2}} \\
	&\quad \, + t_j \sup_{k \geq 1} C_k(t_js) \int_{0}^{\infty}  \vert\partial_s u(s)\vert \text{ d}s 
	+ \sup_{k \geq 1} C_k(t_js) \left( \lVert u \rVert_{W} + \lVert u \rVert_{W}^2 \right).
	\end{align*}
	Using \eqref{selfsim-Ctj_Linfty} and \eqref{selfsim-Ctj_L1}, the above can be estimated by
	\begin{align*}
	s^{-\frac{1}{2}} &\left( \int_{0}^{\infty} (s+1)^2(\partial_s^2 u(s))^2 \text{ d}s \right)^{\frac{1}{2}} 
	+ \frac{1}{s} \left( \int_{0}^{\infty}  (\partial_s u(s))^2(s + 1)^2 \text{ d}s \right)^{\frac{1}{2}} 
	+ \frac{1}{s} \left( \lVert u \rVert_{W} + \lVert u \rVert_{W}^2 \right) \\
	&\lesssim \left(\frac{1}{\sqrt{a}} + \frac{1}{a} \right)\left( \lVert u \rVert_{W} + \lVert u \rVert_{W}^2 \right)
	\end{align*}
	for all $s \in [a,T]$. All in all, we now have
	\begin{equation*}
	\Big\vert \partial_s  \int_{0}^{\infty} \Ctj(s,x)u(x) \text{ d}x \Big\vert
	\leq \left\vert I\right\vert + \left\vert II\right\vert
	\lesssim \left(\frac{1}{\sqrt{a}} + \frac{1}{a} +  \frac{1}{a^2} \right) \left( \lVert u \rVert_{W} + \lVert u \rVert_{W}^2 \right),
	\end{equation*} 
	implying that  for all $u \in W$ the functions $\left\{s \mapsto \int_{0}^{\infty} \Ctj(s,x)u(x) \text{ d}x\right\}_{j \geq 1}$ are equicontinuous on $[a,T]$.
	
	Next, we check that for all $s \in [a,T]$ the set
	$\left\{u \mapsto  \int_{0}^{\infty} \Ctj(s,x)u(x) \text{ d}x\right\}_{j \geq 1} \subset W^*$
	is sequentially weakly relatively compact. This follows easily from the reflexivity of $W^*$ and the fact that
	\begin{align*}
	\Big\vert \int_{0}^{\infty} \Ctj(s,x)u(x) \text{ d}x \Big\vert
	&\leq  \rho^{\frac{1}{2}}s^{-\frac{1}{2}}\lVert u \rVert_{W} \leq \rho^{\frac{1}{2}}a^{-\frac{1}{2}}\lVert u \rVert_{W}
	\end{align*}
	for $s \in [a,T]$. Hence the requirements of Theorem \ref{selfsim-weak-arzela-ascoli} are satisfied and we obtain a function $\hat{C} \in C([a,T],W^*_w)$ and a subsequence, again labelled $j$, such that  $\Ctj \to \hat{C} \in C([a,T],W^*_w)$. By a diagonal argument, we construct a subsequence converging in $C([a,T],W^*_w)$ for all $a > 0$ to a function $\hat{C}$. This implies that for all $u \in W$, the convergence
	\begin{equation}\label{selfsim-AA-limit-to-dual-pairing}
	\int_{0}^{\infty} \Ctj(s,x)u(x) \text{ d}x = \langle \Ctj(s, \cdot), u \rangle \to \langle \hat{C}(s), u \rangle
	\end{equation}
	holds uniformly on compact sets $[a, T]$. 
	
	Our next step is to identify the limit $\hat{C}$. The key tool is that we have already identified the weak* $L^\infty$ limit $\bar{C}(s,x) = \frac{(\rho - \frac{1}{\kappa})}{s}e^{-\frac{x}{s}}$. Let $u \in W$ and $\zeta \in C_c^\infty((0,T))$. Then on the one hand, since $\zeta$ has compact support, by \eqref{selfsim-AA-limit-to-dual-pairing} it holds that
	\begin{equation*}
	\lim\limits_{j \to \infty} \int_{0}^{T} \int_{0}^{\infty} \Ctj(s,x) u(x) \zeta(s) \text{ d}x \text{ d}s = \int_{0}^{T} \zeta(s) \langle \hat{C}(s, \cdot), u \rangle \text{ d}s.
	\end{equation*}
	On the other hand, since $W \subset L^1((0, \infty))$ and therefore $u \zeta \in L^1([a,T) \times (0, \infty))$ for some $a > 0$, we have by \ref{thm-weak*-limit_sx} that
	\begin{equation*}
	\lim\limits_{j \to \infty} \int_{0}^{T} \int_{0}^{\infty} \Ctj(s,x) u(x) \zeta(s) \text{ d}x \text{ d}s = \int_{0}^{T} \int_{0}^{\infty} \zeta(s) u(x) \frac{(\rho - \frac{1}{\kappa})}{s} e^{-\frac{x}{s}} \text{ d}x \text{ d}s.
	\end{equation*}
	This implies that 
	\begin{equation}\label{selfsim-eqn-for-hat-C-AA-limit}
	\langle \hat{C}(s, \cdot), u \rangle = \int_{0}^{\infty} u(x) \frac{(\rho - \frac{1}{\kappa})}{s} e^{-\frac{x}{s}} \text{ d}x
	\end{equation}
	for all $s \in (0,T)$ except on a set of measure zero, and since
	$s \mapsto \langle \hat{C}(s, \cdot), u \rangle$
	and
	$s \mapsto \int_{0}^{\infty} u(x) \frac{1}{2s} e^{-\frac{x}{s}} \text{ d}x$
	are both continuous on $(0,T)$, the equality (\ref{selfsim-eqn-for-hat-C-AA-limit}) holds for all $s \in (0,T)$. Hence \eqref{selfsim-AA-limit-to-dual-pairing} becomes
	\begin{equation*}
	\int_{0}^{\infty} \Ctj(s,x)u(x) \text{ d}x \to \int_{0}^{\infty} u(x) \frac{(\rho - \frac{1}{\kappa})}{s} e^{-\frac{x}{s}} \text{ d}x
	\end{equation*}
	uniformly on compact intervals in $(0,T)$.
	In particular, for $s = 1$ we obtain
	\begin{equation}
	\int_{0}^{\infty} \Ctj(1,x)u(x) \text{ d}x \to \int_{0}^{\infty} u(x) \Big(\rho - \frac{1}{\kappa}\Big) e^{-x} \text{ d}x
	\end{equation}
	for all $u \in W$.
\end{proof}
We can strengthen this result by applying the Banach-Alaoglu theorem to $\Ctj(1,x) \in L^\infty([0, \infty))$, yielding weak* convergence in $L^\infty$ to the same limit.
\begin{corollary}\label{selfsim-cor-s=1}
	There exists a subsequence $j \to \infty$ such that $\Ctj(1, x) \weakstar (\rho - \frac{1}{\kappa})e^{-x}$ in $L^\infty([0, \infty))$.
\end{corollary}
\begin{proof}
	By \eqref{selfsim-Ctj_Linfty}, we have $\lVert \Ctj(1, \cdot) \rVert_{L^\infty([0, \infty))} \lesssim 1$, so by the Banach-Alaoglu theorem there exists a subsequence converging in the weak* sense in $L^\infty([0, \infty))$. By Theorem \ref{selfsim-main theorem} and testing against functions $u \in C_c^\infty$, the weak* limit must be equal almost everywhere to $(\rho - \frac{1}{\kappa})e^{-x}$.
\end{proof}

We can further improve the convergence locally by combining the previously established bound on the time derivatives $\partial_s \Ctj(s,x) \in W^*$ for $s \in [a,T]$ with a bound on the total variation.
\begin{theorem}\label{thm-sa-L^p_loc-convergence}
	Let $0 < a < T < \infty$. There exists a subsequence, again denoted by $t_j$, such that for all $1 \leq p < \infty$ and $M < \infty$ it holds that
	$$\Ctj(s,x) = t_jC(t_js, \lfloor t_j x \rfloor + 1) \to \frac{(\rho - \frac{1}{\kappa})}{s}e^{-\frac{x}{s}}$$
	in $C([a,T], L^p((0,M)))$.
\end{theorem}
\begin{proof}
	Let $0 < M < \infty$ and $1 \leq p < \infty$. We use an improvement of the Aubin-Lions lemma due to Simon \cite[Cor. 4]{simon} with the spaces $\BV((0,M))$, $L^p((0,M))$ and $W^*$, where $W$ now denotes
	$W \coloneqq W^{2,2}((1+x)^2\text{ d}x, (0,M))$.
	The space $\BV((0,M)) $ embeds compactly into $L^p((0,M))$ for all $1 \leq p < \infty$ \cite[Cor. 3.49]{ambrosio-bv}, and the embedding $L^p((0,M)) \hookrightarrow W^*$ is continuous for all $1 \leq p < \infty$ since
	$W \subset C^1((0,M)) \subset L^q((0,M))$
	for all $1 \leq q \leq \infty$.
	
	For all $s > 0$ and $j \in \N$, it holds that $\Ctj(s, \cdot)$ is monotonically decreasing and bounded, so the family $\Ctj \subset L^\infty([a,T], \BV((0,M)))$ is bounded. Combined with the previously established uniform  bound in $s \in [a,T]$ on $\partial_s \Ctj(s,x) \subset W^*$, this completes the assumptions of the Aubin-Lions lemma. We obtain a function $\hat{C} \in C([a,T], L^p((0,M)))$ and a subsequence denoted again by $t_j$ such that
	$\Ctj \to \hat{C}$ in $C([a,T], L^p((0,M)))$. 
	
	To identify the limit $\hat{C}$, we use that strong convergence implies weak convergence and observe that for all $u \in L^{\frac{p}{p-1}}((0,M))$, respectively $u \in L^\infty((0,M))$ when $p = 1$, and $\zeta \in C_c^\infty((a,T))$ it holds
	\begin{equation*}
	\int_{a}^{T} \int_{0}^{M} \Ctj(s,x)u(x)\zeta(s) \text{ d}x \text{ d}s \to \int_{a}^{T} \int_{0}^{M} \frac{(\rho - \frac{1}{\kappa})}{s}e^{-\frac{x}{s}} u(x)\zeta(s) \text{ d}x \text{ d}s,
	\end{equation*}
	so as before by continuity of $\hat{C}$ we have for all $s \in [a,T]$ that $\hat{C}(s,x) = \frac{(\rho - \frac{1}{\kappa})}{s}e^{-\frac{x}{s}}$ almost everywhere.
	
	We thus obtain for all $1 \leq p < \infty$ and $M < \infty$ a subsequence $t_j$ such that $\Ctj \to \frac{(\rho - \frac{1}{\kappa})}{s}e^{-\frac{x}{s}}$ in $C([a,T], L^p((0,M)))$. By a diagonal argument, we obtain a single subsequence $t_j$ such that this convergence holds in $C([a,T], L^p((0,M)))$ for all $1 \leq p < \infty$ and $M < \infty$.
\end{proof}

Since the limits in Corollary \ref{selfsim-cor-s=1} and Theorem \ref{thm-sa-L^p_loc-convergence} are uniquely identified, we in fact have convergence for every sequence $t \to \infty$. 
Furthermore, as a consequence of convergence in $L^p_{loc}$ we also obtain the pointwise convergence
$t_jC(t_j, \lfloor t_j x \rfloor + 1) \to (\rho - \frac{1}{\kappa})e^{-x}$
almost everywhere along a subsequence.

\bibliographystyle{alpha}
\bibliography{paper}

\end{document}